\documentclass[11pt, reqno]{amsart}
\usepackage[utf8]{inputenc}
\usepackage{amsmath,bm}
\usepackage{amsthm}
\usepackage{amssymb}
\usepackage{amsfonts}
\usepackage{MnSymbol}
\usepackage{caption}
\usepackage{hyperref}
\usepackage{blindtext,amsmath,amsthm,amsfonts, textcomp, mathrsfs, mathtools}
\usepackage[shortlabels]{enumitem}

\usepackage[top=30mm,bottom=30mm,left=30mm,right=30mm]{geometry}
\newtheorem{theorem}{Theorem}[section]

\newtheorem*{theorem*}{Theorem}
\newtheorem*{remark*}{Remark}
\newtheorem*{problem*}{Problem}
\newtheorem*{conjecture*}{Conjecture}
\newtheorem*{question*}{Question}
\newtheorem{lemma}[theorem]{Lemma}

\newcommand{\rom}[1]{\uppercase\expandafter{\romannumeral #1\relax}}
\newcommand{\Q}{\mathbb{Q}}

\newcommand{\R}{\mathbb{R}}
\newcommand{\C}{\mathbb{C}}

\newcommand{\ra}{\mathfrak{a}}

\def\house#1{{%
    \setbox0=\hbox{$#1$}
    \vrule height \dimexpr\ht0+1.4pt width .4pt depth \dp0\relax
    \vrule height \dimexpr\ht0+1.4pt width \dimexpr\wd0+2pt depth \dimexpr-\ht0-1pt\relax
    \llap{$#1$\kern1pt}
    \vrule height \dimexpr\ht0+1.4pt width .4pt depth \dp0\relax
}}

\begin{document}

\title[On the lowest zero of Dedekind zeta function]{On the lowest zero of Dedekind zeta function}

\author[Sushant Kala]{Sushant Kala}

\address{Department of Mathematics\\ Institute of Mathematical Sciences (HBNI)\\ CIT Campus, IV Cross Road\\ Chennai\\ India-600113}
\email{sushant@imsc.res.in}
\date{}

\begin{abstract}
Let $\zeta_K(s)$ denote the Dedekind zeta-function associated to a number field $K$. In this paper, we give an effective upper bound for the height of first non-trivial zero other than $1/2$ of $\zeta_K(s)$ under the generalized Riemann hypothesis. This is a refinement of the earlier bound obtained by Omar Sami. 
\end{abstract}



\maketitle

\section{\bf Introduction}

Let $K/\Q$ be a number field. The Dedekind zeta-function associated with $K$ is defined on $\Re(s)>1$ as
$$
    \zeta_K(s) := \sum_{\ra} \frac{1}{N\ra^s}.
$$
Here $\ra$ runs over all non-zero integral ideals of $K$. This function has an analytic continuation to $\C$ except for a  simple pole at $s=1$. The zeros of $\zeta_K(s)$ in the critical strip $0<\Re(s) < 1$ are called the non-trivial zeros. One of the central problems in analytic number theory is to study the order and magnitude of these non-trivial zeros. The Generalized Riemann Hypothesis (GRH) says that all the non-trivial zeros of $\zeta_K(s)$ lie on the vertical line $\Re(s) = \frac{1}{2}$. Under GRH, one can consider the height of a zero, i.e., its distance from the point $s=1/2$. Define
\begin{equation*}
    \tau(K):= \min\{t>0, \zeta_K(1/2 + it) =0\},
\end{equation*}
the lowest height of a non-trivial zero of $\zeta_K(s)$ other than $1/2$. It is possible that $\zeta_K(\frac{1}{2})=0$, as shown by J. V. Armitage \cite{Armitage} in 1971. However, it is believed that as we vary over number fields, $\zeta_K(\frac{1}{2})$ vanishes very rarely. Indeed, K. Soundararajan \cite{Sound} showed that for a large proportion (87.5 \%) of quadratic number fields, $\zeta_K(\frac{1}{2}) \neq 0$.\\

One of the natural questions is to obtain upper and lower bounds on $\tau(K)$. The importance of studying $\tau(K)$ is evident from its connection to the discriminant of the number field, as highlighted in the survey paper by A. M. Odlyzko \cite{Odlyzko}. Furthermore, the low-lying zeros of $\zeta_K(s)$ also have repercussions to Lehmer's conjecture on heights of algebraic numbers (see \cite{Dixit-Sushant}). In 1979, J. Hoffstein \cite{Hoffstein} showed that for number fields $K$ with sufficiently large degree
$$
    \tau(K) \leq 0.87.
$$

\medskip
\noindent
For a number field $K$, denote by $n_K$ the degree $[K:\Q]$ and $d_K$ the discriminant $disc(K/\Q)$. Let $\alpha_K$ be the log root discriminant of $K$ defined as 
$$
\alpha_K : =\frac{\log |d_K|}{n_K}.
$$
In 1985, A. Neugebauer \cite{NG_1} showed the existence of a non-trivial zero of $\zeta_K(s)$ in the rectangle 
$$
\mathfrak{R}=\left\{ \,\, \sigma + i t \,\,\, | \,\,\, \frac{1}{2} \leq \sigma \leq 1,\,\,\, |t-T| \leq 10 \,\,\right\},
$$
\noindent
for every $T\geq 50$. Later in 1988, Neugebauer \cite{NG_2} derived an explicit upper bound, namely either  $\zeta_K(1/2)=0$ or
\begin{equation}\label{Neugebaur}
\tau(K) \leq \operatorname{min} \left\{ 60, \frac{64 \pi^2}{\log \big( \frac{1}{4} \log (82 + 27 \alpha_K) \big)}  \right\}.
\end{equation}
\medskip

Conjecturally, E. Tollis \cite{Tollis} asserts that
\begin{equation}\label{tollis}
    \tau(K) \ll \frac{1}{\log |d_K|},
\end{equation}
where the implied constant is absolute. Although this remains open, O. Sami \cite{Sami} showed that under GRH,
$$
\tau(K) \ll_{n_K} \frac{1}{\log \log \left(\left|d_K\right|\right)}.
$$
Thus, the lowest zero of the Dedekind zeta function converges to $\frac{1}{2}$ as we vary over number fields with a fix degree.\\

Let $\tau_0:= \tau(\Q) (=14.1347\ldots) $ be the lowest zero of the Riemann zeta-function $\zeta(s)$. Recall the famous Dedekind's conjecture, which states that $\zeta_K(s)/\zeta (s)$ is entire. Therefore, one expects $\zeta_K(1/2 + i\tau_0)=0$ for all number fields $K$. We obtain the following effective bound for the lowest zero of $\zeta_K(s)$.

\begin{theorem}\label{Main Theorem}
 Let $K$ be a number field such that the log root discriminant $\alpha_K > 6.6958$ and $\zeta_K(1/2) \neq 0$. Then, under GRH either $\tau(K) \geq \tau_0$ or
\medskip

\noindent
     $$
      \tau(K) \leq \frac{ \pi}{ \sqrt{2} \log \Big( \frac{\alpha_K - 1.2874}{5.4084} \Big)}.
     $$
\end{theorem}

\medskip

\textbf{Remark.} One can improve this bound using Hoffstein's result \cite[pp. 194]{Hoffstein}, which states that $\tau(K) \leq 0.87$ for all number fields with sufficiently large degree. Indeed, the method of our proof shows that for number fields $K$ with sufficiently large degree, if $\alpha_K > 6.4435$, then under GRH 
\begin{equation*}
      \tau(K) \leq \frac{ \pi}{ \sqrt{2} \log \Big( \frac{\alpha_K - 1.2874}{5.1561} \Big)}.
\end{equation*}



\medskip

Next, we address the case where $\zeta_K(s)$ vanishes at $s=1/2$.

\begin{theorem}\label{Main Theorem 2}
     Suppose $K$ is a number field with $\alpha_K > 12.1048 $ and $\zeta_K(1/2) = 0$. Let
     
     $$
         A :=\frac{\alpha_K - 1.2874}{2\,\, \left( \frac{17.2}{\pi^2} \frac{\alpha_K}{ \log \log |d_K|} \right)} \text{   and  } B := \log \left( \frac{\alpha_K - 1.2874}{10.8168}\right).
     $$
\medskip

     \noindent
     Then, under GRH, either $\tau(K) \geq \tau_0$ or
     
     $$
      \tau(K) \leq \frac{\sqrt{2} \pi}{ \operatorname{min}\{ A, B \}}.
     $$
\end{theorem}

From Tollis's conjecture \eqref{tollis}, it is clear that over any family of number fields $\{K_i\}$, the height of the lowest zero $\tau(K)$ tends to $0$. However, in Theorem \ref{Main Theorem} and \ref{Main Theorem 2} (also in \cite{Sami}), we show this for families of number fields $\{K_i\}$, where the root discriminant tends to infinity. This property is also discussed in \cite[Proposition 5.2]{Tsfasman}. Also note that the bound in Theorem \ref{Main Theorem 2} is weaker than Theorem \ref{Main Theorem}. This is perhaps indicative of the curious phenomena that a zero at $\frac{1}{2}$ pushes the next zero away from itself.
\medskip

\section{\bf Preliminaries}
\medskip

In this section, we state and prove some results which will be useful in the proof of main theorems.\\

\noindent
We first recall Weil's explicit formula. Let $F$ be a real valued even function satisfying the following conditions:

\begin{enumerate}[(i)]
    \item $F$ is continuously differentiable on $\mathbb{R}$ except at a finite number of points $a_i$ where $F(x)$ and its derivative $F^{\prime}(x)$ has only discontinuities of the first kind for which $F$ satisfies the mean condition, i.e. 
    $$
        F\left(a_i\right)=\frac{1}{2}\left(F\left( a_i+0\right)+F\left(a_i-0\right)\right).
    $$

    \item There exists $b>0$ such that $F(x)$ and $F^{\prime}(x)$ are $O\left(e^{-(1 / 2+b)|x |}\right)$ in the vicinity of $\infty$.
\end{enumerate}
\noindent
Then, the Mellin transform of $F$, given by 
$$
\Phi(s) :=\int_{-\infty}^{\infty} F(x) e^{(s-1 / 2) x} d x
$$
is holomorphic in any strip $-a \leq \sigma \leq 1+a$, where $0<a<b$, $a<1$. Then, we have the following explicit formula due to Weil \cite{weil} (formulated by Poitou).

\begin{theorem}[Weil]\label{explicit formula}
Let $F$ satisfy conditions (i) and (ii) above with $F(0)=1$. Then the sum $\sum \Phi(\rho)$ taken over the non-trivial zeros $\rho=\beta+i \gamma$ of $\zeta_K(s)$ with $|\gamma|<T$ has a limit when $T$ tends to infinity and its sum is given by the formula
\begin{align}\label{Weil}
\sum_{\rho} \Phi(\rho)= & \Phi(0)+\Phi(1)-2 \sum_{\mathfrak{p}} \sum_{m=1}^{\infty} \frac{\ln (N(\mathfrak{p }))}{N(\mathfrak{p})^{m / 2}} F(m \ln (N(\mathfrak{p}))) + \ln \left(\left|d_K\right|\right) \nonumber\\ 
&- n_K [\ln (2 \pi)+\gamma+2 \ln (2)]-r_1 J(F)+n I(F),
\end{align}
with
$$
J(F)=\int_0^{\infty} \frac{F(x)}{2 \operatorname{cosh}(x / 2)} d x, \quad I(F)=\int_0^{\infty} \frac{1-F(x)}{2 \operatorname{sinh}(x / 2)} d x
$$
and $\gamma=0.57721566 \ldots$ denotes the Euler-Mascheroni constant. Here $\mathfrak{p}$ runs over all the prime ideals of $K$ and $N(\mathfrak{p})$ denote the ideal norm of $\mathfrak{p}$.

\end{theorem}

\medskip

Note that
$$
\Phi(0)+\Phi(1)=4 \int_0^{\infty} F(x) \operatorname{cosh}(x / 2) d x \text {. }
$$
If $\widehat{F}$ denotes the Fourier transform of $F$, then under GRH, we have $\Phi(\rho)=\widehat{F}(t)$, where $\rho=1 / 2+i t$. Set $F_T(x):= F\left(\frac{x}{T}\right)$, then $\widehat{F}_T(u)=T \widehat{F}(T u)$. We now recall two lemmas proved in \cite{Sami}.

\begin{lemma}\label{Lemma 1}(Sami)
Let $F$ be a compactly supported even function defined on $\R$ as
$$
F(x)= \begin{cases}(1-|x|) \cos (\pi x)+\frac{3}{\pi} \sin (\pi |x|) & \text { if } 0 \leq |x| \leq 1, \\ 0 & \text { else. }\end{cases}
$$
Then $F$ satisfies the growth conditions of explicit formula and
$$
\widehat{F}(u)=2 \left(2-\frac{u^2}{\pi^2}\right)\left[\frac{2 \pi}{\pi^2-u^2} \cos(u/2)\right]^2 .
$$
\end{lemma}

\begin{lemma} \label{Lemma 3}(Sami)
Let $a, b, c$ be three positive real constants satisfying $c>$ $2 b$. If $T>0$ and $a T+b e^{T / 2} \geq c$, then
\begin{equation*}
T \geq \min \left(\frac{c}{2 a}, \ln \left( \frac{c}{2b} \right) \right) .
\end{equation*}
\end{lemma}

\medskip

\section{\bf Proof of main theorems}

\medskip

The proof of our theorems follows a similar method as in \cite{Sami}. We start with the following lemma.

\begin{lemma} Let $F_T(x)=F\left(\frac{x}{T}\right)$ as in explicit formula (\ref{Weil}). Then we have the following estimate.

\begin{equation}
  \sum_{\mathfrak{p}}\sum_{m=1}^{\infty} \frac{\log (N(\mathfrak{p}))}{N(\mathfrak{p})^{m / 2}} F_T(m \log (N(\mathfrak{p}))) \leq 1.2571\,\, n_K \,(2\,e^{T/2}-1),
\end{equation}
where the implied constant is absolute and $\mathfrak{p}$ runs over all prime ideals of $K$.
\end{lemma}
\begin{proof}
 Let $p$ be a rational prime. Since $
  \sum_{p | \mathfrak{p}} \log N(\mathfrak{p}) \leq n_K \log p$, we have  
  $$
  \sum_{p | \mathfrak{p}} \frac{\log N(\mathfrak{p})}{N(\mathfrak{p})^{m/2}} \leq n_K \frac{\log p}{p^{m/2}}.
 $$
 From the definition of $F(x)$, it can be obtained that $|F(x)| \leq 1.21$. Hence, the above inequality gives

\begin{align}\label{MT_1} \sum_{\mathfrak{p}, m} \frac{\log N(\mathfrak{p})}{N(\mathfrak{p})^{m / 2}} F_T(m \log N(\mathfrak{p})) \,\, &= \,\,
\sum_{m,p} \sum_{p | \mathfrak{p}} \frac{\log N(\mathfrak{p})}{N(\mathfrak{p})^{m / 2}} F_T(m \log N(\mathfrak{p})) \nonumber \\
&\leq \,\, 1.21 \, n_K \sum_{m \log p \leq T} \frac{\log p}{p^{m / 2 }} \nonumber \\
&=\,\,  1.21 \, n_K \sum_{n \leq e^T} \frac{\Lambda(n)}{\sqrt{n}} ,
\end{align}
where $\Lambda$ is the von Mangoldt function. Now, recall the Chebyshev function 
$$
\Psi(x) := \sum_{n \leq x} \Lambda(n).
$$ 
Applying partial summation and using the bound $\Psi(x) \leq 1.0389 \,x$ by Rosser \cite{Rosser}, we have

\begin{align}\label{MT_2}
\sum_{n \leq e^T} \frac{\Lambda(n)}{\sqrt{n}} \, \, &= \, \, \frac{\Psi(e^T)}{e^{T/2}} + \frac{1}{2} \int_1^{e^T} \frac{\Psi(t)}{t^{3/2}}dt \nonumber \\
&\leq 1.0389 \, \,  \Big( 2 e^{T/2} - 1 \Big).
\end{align}

\medskip

\noindent
From \eqref{MT_1} and \eqref{MT_2}, the lemma follows.
\end{proof}

\bigskip

Let
$T=\frac{\sqrt{2} \pi}{ \tau(K)}$ and $F(x)$ be the function defined in Lemma \ref{Lemma 1}.  Applying Theorem \ref{explicit formula} to $F_T(x)=F(x / T)$, we get

\begin{equation}
\begin{aligned}\label{Exp_main}
\sum_{\rho} \Phi(\rho) \,\,\, &= \,\,\,  \Phi_T(0)+\Phi_T(1)-2 \sum_{\mathfrak{p}, m} \frac{\log (N(\mathfrak {p}))}{N(\mathfrak{p})^{m / 2}} F_T(m \log (N(\mathfrak{p}))) \\
& +\log |d_K|-n_K[\log (2 \pi)+\gamma+2 \log (2)]-r_1 J\left(F_T\right)+n_K I\left(F_T\right).
\end{aligned}
\end{equation}

\medskip

Since $\tau(K) \leq \tau_0$, we have $T \geq 0.314$. For such $T$, the remaining terms on the right-hand side of (\ref{Exp_main}) can be bounded as 
\begin{align}
    J(F_T) &\,\,\,\,=\,\,\,\, \int_0^{T} \frac{F(x/T)}{2 \operatorname{cosh}(x / 2)} d x \,\,\,\,\leq\,\,\,\, 0.276 \,\, e^{T/2} \label{bound_2}
\end{align}

\medskip

\noindent
and
\begin{align}\label{bound_3}
    I(F_T) &\,\,\,\,=\,\,\,\, \int_0^{T} \frac{1-F(x/T)}{2 \operatorname{sinh}(x / 2)} d x \,\,\,\, \geq \,\,\,\, -0.1034 \,\, e^{T/2}. 
\end{align}



\bigskip
\noindent
We are now ready to prove our theorems.

\subsection{Proof of Theorem \ref{Main Theorem}}
\begin{proof}
\noindent
Since $\zeta_K(1/2) \neq 0$, equation (\ref{Exp_main}) gives
\begin{align*}
\log |d_K| + \Phi_T(0)+\Phi_T(1) &\leq 2 \sum_{\mathfrak{p}, m} \frac{\log (N(\mathfrak {p}))}{N(\mathfrak{p})^{m / 2}} F_T(m \log (N(\mathfrak{p}))) + n_K[\log (2 \pi)+\gamma+2 \log (2)] \\
&+ r_1 J\left(F_T\right)-n_K I(F_T).
\end{align*}
\noindent
Using Lemma \ref{Lemma 1} along with (\ref{bound_2}) and (\ref{bound_3}), we deduce
\begin{align*}
    \log |d_K| \,\,\,\, &\leq \,\,\,\, 5.4084 \,\, n_K \, e^{T/2} + 1.2874 \,\, n_K.
\end{align*}
\noindent
Thus,
\begin{align*}
   \alpha_K - 1.2874 \leq 5.4084 \,\,e^{T/2}.
\end{align*}
Hence, for $\alpha_K > 6.6958$ 
$$
T \geq 2 \log \Big( \frac{\alpha_K - 1.2874}{5.4084} \Big).
$$
\noindent
Since $T=\frac{\sqrt{2}\pi}{\tau(K)}$, the Theorem follows.
\end{proof}

\bigskip


\subsection{Proof of Theorem \ref{Main Theorem 2} }
\medskip

\begin{proof}
\noindent
Here $\zeta_K(\frac{1}{2})=0$ and therefore equation (\ref{Exp_main}) gives

\begin{align*}
\log |d_K| + \Phi_T(0)+\Phi_T(1) &\leq 2 \sum_{\mathfrak{p}, m} \frac{\log (N(\mathfrak {p}))}{N(\mathfrak{p})^{m / 2}} F_T(m \log (N(\mathfrak{p}))) + n_K[\log (2 \pi)+\gamma+2 \log (2)] \\
&+ r_1 J\left(F_T\right)-n_K I(F_T) + \frac{16}{\pi^2}\,r\,T,
\end{align*}
where $r$ is the order of $\zeta_K(s)$ at $1/2$. As before, using Lemma \ref{Lemma 1} along with (\ref{bound_2}) and (\ref{bound_3}), we deduce
\begin{align*}
    \log |d_K| \,\,\,\,&\leq\,\,\,\, 5.4084 \,\, n_K \, e^{T/2} + 1.2874 \,\, n_K + \frac{16}{\pi^2} \,\,r T .
\end{align*}
We now use the following bound on order of zero of $\zeta_K(s)$ at $s=1/2$ (see \cite[Proposition 1]{Sami}),
$$
r \leq \frac{\log |d_K|}{\log \log |d_K|} + \frac{n_K}{2 \log \log |d_K|}.
$$
Thus
\begin{align*}
    \alpha_K - 1.2874 \,\,\,\, & \leq 5.4084 \,\, e^{T/2} +  \Big( \frac{17.2}{\pi^2}\frac{\alpha_K} {\log \log |d_K|} \Big) T.
\end{align*}
\noindent
Using Lemma \ref{Lemma 3} with $a=\Big( \frac{17.2}{\pi^2}\frac{\alpha_K} {\log \log |d_K|} \Big), \, b=5.4084$ and $c=\alpha_K - 1.2874$, we conclude
  $$
      \tau(K) \leq \frac{\sqrt{2} \pi}{ \operatorname{min}\{ A, B \}},
  $$
where $A,B$ are as in the statement of the theorem. This concludes the proof.



\end{proof}

\section{\bf Computational data and concluding remarks}

\subsection{Computational data}Let $K=\Q(\beta)$ be a number field and $m_{\beta}(x)$ be the minimal polynomial of $\beta$. Using SageMath, we compare the lowest zero and the bounds obtained using Theorem \ref{Main Theorem}(see table \ref{table : 1}).

\begin{table}[h]
 \centering
  \renewcommand{\arraystretch}{1.5} 

\hspace{1cm}

\begin{tabular}{|c|c|c|c|}
\hline
$m_{\beta}(x)$ & $\alpha_K$ & $\tau(K)$  & Bound in Theorem \ref{Main Theorem} \\ 
\hline
$x^2 + 510510$ & 7.26472993307674 & 0.195366057287247 & 22.2098243056698 \\ \hline
$x^2 + 9699690$ & 8.73694942265996 & 0.250485767971509 & 6.93766313396318\\ \hline
$x^2 + 223092870$ & 10.3046965306245 & 0.282126995483731 & 4.34561699877460\\ \hline
$x^2 + 6469693230$ & 11.9883444456178 & 0.223870166465309 & 3.25543786648311\\ \hline
$x^2 + 200560490130$ & 13.7053380478603 & 0.0869456767128933 & 2.67260773966497\\ \hline
$x^3 + 30030$ &7.97191372931969 & 0.249553262973507 & 10.4864035098435\\ \hline
 $x^4 + 30030$ & 9.11875848185292 & 0.0668359001429184 & 6.00093283699129 \\ \hline
\end{tabular}
\caption{Comparing bound in Theorem \ref{Main Theorem} with the height of first zero}
\label{table : 1}
\end{table}

\medskip

On the other hand, comparing Theorem \ref{Main Theorem} with Neugebaur's bound in (\ref{Neugebaur}), observe that although the bound in (\ref{Neugebaur}) is unconditional, it applies only for the cases where $\alpha_K$ is very large  $(>  10^{64849})$. On the other hand, Theorem \ref{Main Theorem} applies for all $K$ with $\alpha_K \geq 6.6958$. 


\subsection{Concluding remarks.} 
The key idea in obtaining an upper bound for lowest zero is to establish a suitable explicit formula and apply it to a suitable test function. Upper bound for lowest zero for automorphic L-functions was obtained in \cite{Sami2}. Effective bounds for such results can also be obtained by following the method in this paper.

\section{\bf Acknowledgments} 
\noindent
I thank my advisor Dr. Anup Dixit for several fruitful discussions and helpful comments on the exposition of this paper. I am grateful to Dr. Siddhi Pathak for pointing out the result of Hoffstein in \cite{Hoffstein}. I also thank Prof. Jeffery Hoffstein for his support and encouragement.

\end{document}